\documentclass[12pt,oneside,a4paper,english,reqno]{amsart}
\usepackage[T1]{fontenc}
\usepackage[latin9]{inputenc}
\usepackage{geometry}
\usepackage{mathrsfs}
\usepackage{amsthm}
\usepackage{amssymb}
\usepackage{mathrsfs}
\usepackage{color}
\usepackage{hyperref}
\hypersetup{
	colorlinks = true,%
	citecolor = blue,
	filecolor=red,%
	linkcolor = [rgb]{0.65,0.0,0.0},%
	anchorcolor = red,
	pagecolor = red,
	urlcolor= [rgb]{0.65,0.0,0.0}
	linktocpage=true,
	pdfpagelabels=true,
	bookmarksnumbered=true,
}
\makeatletter
\numberwithin{equation}{section}
\numberwithin{figure}{section}
\theoremstyle{plain}
\newtheorem{theorem}{\protect\theoremname}[section]
  \theoremstyle{plain}
  \newtheorem{lemma}{\protect\lemmaname}[section]
  
\newtheorem{remark}{Remark}[section]
\def\d{{{\rm d}%
}v_{g}}
\def\ds{\displaystyle}
\newcommand{\eps}{{\varepsilon}}
\makeatletter
\@addtoreset{theorem}{section}
\makeatother

\makeatother

\usepackage{babel}
  \providecommand{\lemmaname}{Lemma}
\providecommand{\theoremname}{Theorem}
\newtheorem{innercustomthm}{\rm \textbf{Theorem}}
\newenvironment{customthm}[1]
{\renewcommand\theinnercustomthm{#1}\innercustomthm}
{\endinnercustomthm}
\usepackage{etoolbox}

\begin{document}

\title{Schrödinger-Maxwell systems on compact Riemannian manifolds}
\author{Csaba Farkas}

\email{farkas.csaba2008@gmail.com \& farkascs@ms.sapientia.ro}

\address{Department of Mathematics and Computer Science, Sapientia University,
	Tg. Mures, Romania \& Institute of Applied Mathematics, Óbuda University,
	1034 Budapest, Hungary}
\maketitle
\begin{abstract}
	In this paper we are focusing to the following Schr\"odinger-Maxwell system:
	\[
	\ \left\{ \begin{array}{lll}
	-\Delta_{g}u+\beta(x)u+eu\phi=\Psi(\lambda,x)f(u) & \mbox{in} & M,\\
	-\Delta_{g}\phi+\phi=qu^{2} & \mbox{in} & M,
	\end{array}\right.\eqno{(\mathcal{SM}_{\Psi(\lambda,\cdot)}^{e})}
	\]
	where $(M,g)$ is a 3-dimensional compact Riemannian manifold without
	boundary, $e,q>0$ are positive numbers, $f:\mathbb{R}\to\mathbb{R}$
	is a continuous function, $\beta\in C^{\infty}(M)$ and $\Psi\in C^{\infty}(\mathbb{R}_{+}\times M)$
	are positive functions. By various variational approaches, existence of multiple solutions of the problem $(\mathcal{SM}_{\Psi(\lambda,\cdot)}^{e})$ is established.
\end{abstract}
\section{Introduction and statement of the main results}
We are concerned with the nonlinear  Schr\"odinger-Maxwell system
\[
\ \left\{ \begin{array}{lll}
-\Delta_{g}u+\beta(x)u+eu\phi=\Psi(\lambda,x)f(u) & \mbox{in} & M,\\
-\Delta_{g}\phi+\phi=qu^{2} & \mbox{in} & M,
\end{array}\right.\eqno{(\mathcal{SM}_{\Psi(\lambda,\cdot)}^{e})}
\]
where $(M,g)$ is a 3-dimensional compact Riemannian manifold without
boundary, $e,q>0$ are positive numbers, $f:\mathbb{R}\to\mathbb{R}$
is a continuous function, $\beta\in C^{\infty}(M)$ and $\Psi\in C^{\infty}(\mathbb{R}_{+}\times M)$
are positive functions. 


From physical point of view, the Schr\"odinger-Maxwell systems 
\begin{equation}\label{eredeti}
\ \left\{
\begin{array}{lll}
-\frac{\hbar^2}{2m}\Delta u+\omega u+e u\phi=f(x,u) & \mbox{in} &
\mathbb{R}^3 , \\
-\Delta\phi=4\pi e u^2 & \mbox{in} & \mathbb{R}^3,
\end{array}%
\right.
\end{equation}
describe the statical behavior of a charged non-relativistic
quantum mechanical particle interacting with the electromagnetic
field. More precisely, the unknown terms $u : \mathbb{R}^3
\to\mathbb{R}$ and $\phi : \mathbb{R}^3 \to \mathbb{R}$ are the
fields associated to the particle and the electric potential,
respectively, the nonlinear term $f$ models the interaction between the
particles and the coupled term $\phi u$ concerns the interaction with the electric field. Note that the quantities $m$, $e$,
$\omega$ and $\hbar$ are the mass, charge, phase, and Planck's
constant. 

In fact, system (\ref{eredeti}) comes from the
evolutionary nonlinear Schr\"odinger equation by using a
Lyapunov-Schmidt reduction.

The Schr\"odinger-Maxwell system (or its
variants) has been the object of various investigations in the last
two decades, the existence/non-existence of positive
solutions, sign-changing solutions, ground states, radial and non-radial solutions
and semi-classical states has been investigated by several authors. Without sake of completeness, we recall in the sequel some important contributions to the study of system (\ref{eredeti}). Benci and Fortunato \cite{BF} considered the case of
$f(x,s)=|s|^{p-2}s$ with $p\in (4,6)$ by proving the existence of
infinitely many radial solutions for (\ref{eredeti}); their main
step relies on the reduction of system (\ref{eredeti}) to the
investigation of critical points of a "one-variable" energy
functional associated with (\ref{eredeti}). 

Based on the idea of
Benci and Fortunato, under various growth assumptions on $f$ further
existence/multiplicity results can be found in Ambrosetti and Ruiz
\cite{Ambrosetti-Ruiz}, Azzolini
\cite{azzolini1}, in \cite{azzolini2} Azzollini, d'Avenia and Pomponio  were concerned with the existence of a positive radial solution to system \eqref{eredeti} under the effect of a general nonlinear term, in \cite{davenia} the existence of a non radially symmetric solution was established when $p\in (4,6)$, by means of a Pohozaev-type identity, d'Aprile and Mugnai \cite{mugani2,mugnai} proved the non-existence of non-trivial solutions to  system (\ref{eredeti}) whenever $f\equiv 0$ or $f(x,s)=|s|^{p-2}s$  and $p\in (0,2]\cup [6,\infty)$, the same authors proved the existence of a non-trivial radial solution to \eqref{eredeti}, for $p\in [4.6)$. Other existence and multiplicity result can be found in the works of Cerami and Vaira \cite{cerami}, Krist\'aly and Repovs \cite{kristalyrepovs}, Ruiz
\cite{ruiz}, Sun, Chen and Nieto \cite{SCN}, Wang and Zhou
\cite{wang}, and references therein.

In the last five years Schr\"odinger-Maxwell systems has been studied on $n-$dimensional {\it compact or non-compact Rieman\-nian mani\-folds} ($2\leq n \leq 5$) by Druet and Hebey \cite{DH}, Farkas and Krist\'aly  \cite{FK16}, Hebey and Wei \cite{HW}, Ghimenti and Micheletti \cite{GHM, GHM2} and Thizy \cite{Thizy-1, Thizy-2}. More precisely, in the aforementioned papers various forms of the system
\begin{equation}\label{eredeti-2}
\ \left\{
\begin{array}{lll}
-\frac{\hbar^2}{2m}\Delta u+\omega u+e u\phi=f(x,u) & \mbox{in} &
M , \\
-\Delta_g\phi+\phi=4\pi e u^2 & \mbox{in} & M,
\end{array}%
\right.
\end{equation}
has been considered, where $(M,g)$ is a Riemannian manifold.

The aim of this paper is threefold. First we consider the system $(\mathcal{SM}_{\Psi(\lambda,\cdot)}^{e})$ with $\Psi(\lambda,x)=\lambda \alpha(x)$, where $\alpha$ is a suitable function and we assume that $f$ is a sublinear nonlinearity (see the assumptions $(f_1)-(f_3)$ below). In this case we prove that if the parameter $\lambda$ is small enough the system $(\mathcal{SM}_{\lambda}^{e})$ has only the trivial solution, while if $\lambda$ is large enough then the system $(\mathcal{SM}_{\Psi(\lambda,\cdot)}^{e})$ has at least two solutions, see Theorem \ref{thm:sublinear1}. It is natural to ask what happens between this two threshold values. In this gap interval we have no information on the number of solutions $(\mathcal{SM}_{\Psi(\lambda,\cdot)}^{e})$; in the case when $q\to 0$ these two threshold values may be arbitrary close to each other. Similar bifurcation type result for a perturbed sublinear elliptic problem was obtained by Krist\'aly, see \cite{Kristaly-JMAA12}.

Second, we consider the system $(\mathcal{SM}_{\Psi(\lambda,\cdot)}^{\lambda})$ with $\Psi(\lambda,x)=\lambda\alpha(x)+\mu_{0}\beta(x),$ where $\alpha$ and $\beta$ are suitable functions. In order to prove a new kind of multiplicity for the system $(\mathcal{SM}_{\Psi(\lambda,\cdot)}^{\lambda})$ (i.e., $e=\lambda$), we show that certain properties of the nonlinearity, concerning the set of all global minima's, can be reflected to the energy functional associated to the problem, see Theorem \ref{thm:m-megoldas}. 

Third, as a counterpart of Theorem \ref{thm:sublinear1} we will consider the system $(\mathcal{SM}_{\Psi(\lambda,\cdot)}^{e})$ with $\Psi(\lambda,x)=\lambda$, and $f$ here satisfies the so called Ambrosetti-Rabinowitz condition. This type of result is motivated by the result of Anello \cite{Anello} and Ricceri \cite{Ricceri2}, where the authors studied the classical Ambrosetti - Rabinowitz problem, without the assumption $\displaystyle \lim_{t\to0}\frac{f(t)}{t}=0,$ i.e., the authors proved that if the nonlinearity $f$ satisfies the so called (AR) condition and a subcritical growth condition, then if $\lambda$ is small enough the problem 
\[
\ \left\{ \begin{array}{lll}
-\Delta u=\lambda f(u) & \mbox{in} & \Omega,\\
u=0& \mbox{on} & \partial \Omega,
\end{array}\right.
\]
has at least two weak solutions in $H^1_0(\Omega)$.

In the sequel we present precisely our results.
As we mentioned before, we first consider a continuous function $f:[0,\infty)\to\mathbb{R}$
which verifies the following assumptions: 
\begin{itemize}
	\item[$(f_{1})$] \label{f1-felt} $\frac{f(s)}{s}\to0$ as $s\to0^{+}$; 
	\item[$(f_{2})$] \label{f2-felt} $\frac{f(s)}{s}\to0$ as $s\to\infty$; 
	\item[$(f_{3})$] \label{f3-felt} $F(s_{0})>0$ for some $s_{0}>0$, where $\ds F(s)=\int_{0}^{s}f(t){\rm d}t,$
	$s\geq0.$ 
\end{itemize}
Due to the assumptions $(f_{1})-(f_{3})$, the numbers $$\ds c_{f}=\max_{s>{0}}\frac{f(s)}{s}$$ and $$c_F=\max_{s>0}\frac{4F(s)}{2s^2+eqs^4}$$
are well-defined and positive. Now, we are in the position to state
the first result of the paper. In order to do this, first we recall the definition of the weak solutions of the problem $(\mathcal{SM}_{\lambda}^{e})$:
The pair $(u,\phi)\in H_g^1(M)\times H^1_g(M)$ is a {\it  weak
	solution} to the system $(\mathcal{SM}^e_\lambda)$ if
\begin{equation}\label{schrodingerweak}
\int_M (\langle\nabla_g u , \nabla_g v\rangle+\beta(x)uv+e u \phi v )\d=
\int_M \Psi(\lambda,x)f(u)v\d \hbox{ for all } v\in H^1_g(M),
\end{equation}
\begin{equation}\label{maxwellweak}
\int_M (\langle\nabla_g \phi, \nabla_g \psi\rangle +\phi \psi
)\d=q\int_M u^2 \psi\d \hbox{ for all }\psi \in H^1_g(M).
\end{equation}

Our first result reads as follows:
\begin{theorem}
	\label{thm:sublinear1}Let $(M,g)$ be a $3-$dimensional compact
	Riemannian manifold without boundary, and let $\beta \equiv 1$. Assume that $\Psi(\lambda,x)=\lambda\alpha(x)$
	and $\alpha\in C^{\infty}(M)$ is a positive function. If the continuous
	function $f:[0,\infty)\to\mathbb{R}$ satisfies assumptions \hyperref[f1-felt]{$(f_{1})$}-\hyperref[f3-felt]{$(f_{3})$},
	then 
	\begin{itemize}
		\item[\emph{(a)}]  if ${\displaystyle 0\leq\lambda<c_{f}^{-1}\left\Vert {\alpha}\right\Vert _{L^{\infty}}^{-1}}$,
		system \textup{$(\mathcal{SM}_{\Psi(\lambda,\cdot)}^{\lambda})$} has only the trivial
		solution;
		\item[\textit{\emph{(b)}}]  for every ${\displaystyle \lambda\geq c_{F}^{-1}\left\Vert \alpha\right\Vert_{L^1}^{-1}}$,
		system $(\mathcal{SM}_{\Psi(\lambda,\cdot)}^{\lambda})$ has at least two distinct non-zero,
		non-negative weak solutions in $H_{g}^{1}(M)\times H_{g}^{1}(M)$. 
	\end{itemize}
\end{theorem}

\begin{remark}\rm
	\begin{itemize}
		\item[(a)]Due to \hyperref[f1-felt]{$(f_{1})$}, it is clear that $f(0)=0$, thus we can extend
		continuously the function $f:[0,\infty)\to \mathbb R$ to
		the whole $\mathbb R$ by $f(s)=0$ for $s\leq 0;$ thus,  $F(s)=0$ for
		$s\leq 0$.
		
		\item[(b)]  \hyperref[f1-felt]{$(f_{1})$} and \hyperref[f2-felt]{$(f_{2})$} mean that $f$ is superlinear at the origin
		and sublinear at infinity, respectively. Typical functions which fulfill hypotheses \hyperref[f1-felt]{$(f_{1})$}-\hyperref[f3-felt]{$(f_{3})$} are $$f(s)=\min\left(s^r,s^p\right), \ 0<r<1<p,\ \ s\geq 0$$ or $$f(s)=\ln(1+s^2), \ s\geq 0.$$
		\item[(c)] By a three critical points result of Ricceri \cite{Ricceri}, one can prove
		that the number of solutions of the problem $(\mathcal{SM}_{\Psi(\lambda,\cdot)}^{e})$ for $\lambda>\tilde\lambda$  is stable under
		small nonlinear perturbations $g:\mathbb R\to \mathbb R$ of
		subcritical type, i.e., $g(s)=o(|s|^{2^*-1})$ as $|s|\to \infty$,
		$2^*=\frac{2N}{N-2}, \ N>2.$ 
	\end{itemize}
\end{remark}
In order to obtain new kind of multiplicity result for the system
$(\mathcal{SM}_{\Psi(\lambda,\cdot)}^{\lambda})$ (with the choice $e=\lambda$), instead of the assumption \hyperref[f1-felt]{$(f_{1})$}
we require the following one:
\begin{itemize}
	\item[$(f_{4})$]  \label{f4-felt} There exists $\mu_{0}>0$ such that the set of all global minima
	of the function 
	\[
	{\displaystyle t\mapsto\Phi_{\mu_{0}}(t):=\frac{1}{2}t^{2}-\mu_{0}F(t)}
	\]
	has at least $m\geq2$ connected components. 
\end{itemize}
In this case we can state the following result. 
\begin{theorem}
	\label{thm:m-megoldas}Let $(M,g)$ be a $3-$dimensional compact
	Riemannian manifold without boundary. Let $f:[0,\infty)\to\mathbb{R}$
	be a continuous function which satisfies \hyperref[f2-felt]{$(f_{2})$} and \hyperref[f4-felt]{$(f_{4})$},
	$\beta\in C^{\infty}(M)$ is a positive function. Assume that $\Psi(\lambda,x)=\lambda\alpha(x)+\mu_{0}\beta(x),$
	where $\alpha\in C^{\infty}(M)$ is a positive function. Then for
	every $\ds \tau>\|\beta\|_{L^1(M)}\inf_{t}\Phi_{\mu_{0}}(t)$ there
	exists $\lambda_{\tau}>0$ such that for every $\lambda\in(0,\lambda_{\tau})$
	the problem $(\mathcal{SM}_{\Psi(\lambda,\cdot)}^{\lambda})$ has at least $m+1$ weak solutions, $m$ of
	which satisfy the inequality $$\frac{1}{2}\int_M \left(|\nabla_{g} u|^2+\beta(x)u^2\right)\d-\mu_{0}\int_M \beta(x)F(u)\d<\tau.$$
\end{theorem}
\begin{remark}
 Taking into account the result of Cordaro \cite{Cordaro} and Anello \cite{Anello-perturbacio} one can prove the following: Consider the following system:
 $$
 \begin{cases}
 	-\Delta_{g}u+\alpha(x)u+\lambda\phi u=\alpha(x)f(u)+\lambda g(x,u), & M\\
 	-\Delta_{g}\phi+\phi=qu^{2}, & M
 \end{cases}$$ where $\alpha\in L^{\infty}(M)$ with $essinf\alpha>0$, $f:\mathbb{R}\to\mathbb{R}$ is a continuous function and $g:M\times\mathbb{R}\to\mathbb{R},$ besides being a Carathéodory function, is such that, for some $p>3(=\dim M)$, $\sup_{|s|\leq t}g(\cdot,s)\in L^{p}(M)$ and $g(\cdot,t)\in L^{\infty}(M)$ for all $t\in\mathbb{R}$. If the set $$G_{f}=\left\{ t\in\mathbb{R}:\ \frac{1}{2}t^{2}-\int_{0}^{t}f(s)\mathrm{d}s=\inf_{\xi\in\mathbb{R}}\left(\frac{1}{2}\xi^{2}-\int_{0}^{\xi}f(s)\mathrm{d}s\right)\right\} $$ has $m\geq2$ bounded connected components, then the system has at least $m+\ds \left[\frac{m}{2}\right]$ weak solutions. For the proof, one can use a truncation argument combining with the abstract critical point theory result of Anello \cite[Theorem 2.1]{Anello-abtract}. Note that, the similar truncation method which was presented in \cite{Cordaro} fails, due to the extra term $\ds \int_M \phi_u u^2$. To overcome this difficulty, one can use the same method as in \cite[Proposition 3.1 (i)\&(ii)]{FK16} \rm{(}see also \cite{Kri-JMPA}\rm{)}.
\end{remark}

Note also that, similar multiplicity results was obtained by Krist\'aly and R\v adulescu in \cite{Kristaly-Studia}, for Emden-Fowler type equations.

Our abstract tool for proving the Theorem \ref{thm:m-megoldas} is the following abstract theorem that we recall here (see \cite{Ricceri-tobb}):

\begin{customthm}{\textbf{A}}\label{Ricceri-m}Let $H$ be a separable and reflexive real Banach
	space, and let $\mathcal{N},\mathcal{G}:H\to\mathbb{R}$ be two sequentially
	weakly lower semi-continuous and continuously Gateaux differentiable
	functionals., with $\mathcal{N}$ coercive. Assume that the functional
	$\mathcal{N}+\lambda\mathcal{G}$ satisfies the Palais-Smale condition
	for every $\lambda>0$ small enough and that the set of all global
	minima of $\mathcal{N}$ has at least $m$ connected components in
	the weak topology, with $m\geq2.$ Then, for every $\eta>\inf_{H}\mathcal{N}$,
	there exists $\overline{\lambda}>0$ such that for every $\lambda\in(0,\overline{\lambda})$
	the functional $\mathcal{N}+\lambda\mathcal{G}$ has at least $m+1$
	critical points, $m$ of which are in $\mathcal{N}^{-1}((-\infty,\eta)).$ 
	
\end{customthm}

%

As a counterpart of the Theorem \ref{thm:sublinear1} we consider
the case when the continuous function $f:[0,+\infty)\to\mathbb{R}$
satisfies the following assumptions:
\begin{itemize}
	\item[$(\tilde{f}_{1})$] \label{f222-felt}$|f(s)|\leq C(1+|s|^{p-1})$, for all $s\in\mathbb{R}$,
	where $p\in(2,6)$; 
	\item[$(\tilde{f}_{2})$] \label{f333-felt} there exists $\eta>4$ and $\tau_{0}>0$ such
	that 
	\[
	0<\eta F(s)\leq sf(s),\forall |s|\geq\tau_{0}.
	\]
\end{itemize}
\begin{theorem}
	\label{thm:szuperlinear}Let $(M,g)$ be a $3-$dimensional compact
	Riemannian manifold without boundary, and let $\beta \equiv 1$. Assume that $\Psi(\lambda,x)=\lambda$. Let $f:\mathbb{R}\to\mathbb{R}$ be a
	continuous function, which satisfies hypotheses \hyperref[f222-felt]{{\rm ($\tilde{f}_{1}$)}},
	\hyperref[f333-felt]{\rm ($\tilde{f}_{2}$)}. Then there exists $\lambda_{0}$ such that for
	every $0<\lambda<\lambda_{0}$ the problem $(\mathcal{SM}_{\lambda}^e)$
	has at least two weak solutions.
\end{theorem}

Our abstract tool for proving the previous theorem is the following abstract theorem that we recall here (see \cite{Ricceri2}):
\begin{customthm}{\textbf{B}}
\label{Ricceritetel} Let $E$ be a reflexive real Banach space, and let $\Phi, \Psi :E \to \mathbb{R}$ be two continuously
G\^ateaux differentiable functionals such that $\Phi$ is
sequentially weakly lower semi-continuous and coercive. Further,
assume that $\Psi$ is sequentially weakly continuous. In addition,
assume that, for each $\mu > 0$, the functional $J_\mu:=\mu\Phi -
\Psi$ satisfies the classical compactness Palais-Smale condition.
Then for each $\rho >\inf_E\Phi$ and each $$\mu>\inf_{u\in
	\Phi^{-1}((-\infty,\rho))}\frac{\sup_{v \in
		\Phi^{-1}((-\infty,\rho))}\Psi(v)-\Psi(u)}{\rho-\Phi(u)},$$
the following alternative holds: either the functional $J_\mu$ has a strict
global minimum which lies in $\Phi^{-1}((-\infty,\rho))$, or $J_\mu$ has at least two critical points one of which lies in $\Phi^{-1}((-\infty,\rho))$.
\end{customthm}
%
\section{Proof of the main results}
Let $\beta\in C^\infty(M)$ be a positive function. For every $u\in C^\infty(M)$ let us denote by $$\|u\|_\beta^2=\int_M |\nabla_g u|^2+\beta(x)u^2\d.$$ The
Sobolev space $H^1_\beta$ is defined as the completion of $C^\infty(M)$ with respect
to the norm $\|\cdot\|_\beta$. Clearly, $H^1_\beta$ is a Hilbert space. Note that, since $\beta$ is positive, the norm $\|\cdot\|_\beta$ is equivalent to the standard norm, i.e., we have that 
\begin{equation}\label{normakhasonlitasa}\min\left\{1,\min_M \sqrt{\beta(x)}\right\}\|u\|_{H^1_g(M)}\leq \|u\|_\beta\leq \max\left\{1,\sqrt{\|\beta\|_{L^\infty(M)}}\right\}\|u\|_{H^1_g(M)}.
\end{equation}
Note that, $H^1_\beta(M)$ is compactly embedded in $L^p(M)$, $p\in [1,6)$; the Sobolev embedding constant will be denoted by $\kappa_{p}$.

We define the energy functional $\mathscr{J}_\lambda:H^1_g(M)\times
H^1_g(M)\to \mathbb{R}$ associated with system $(\mathcal{SM}_{\lambda}^e)$, namely,
$$\mathscr{J}_\lambda(u,\phi)=\frac{1}{2}\|u\|_\beta^2+\frac{e}{2}\int_M
\phi u^2 {\rm d}v_g-\frac{e}{4q}\int_M |\nabla_g \phi|^2{\rm
	d}v_g-\frac{e}{4q}\int_M\phi^2 {\rm d}v_g-\int_M
\Psi(x,\lambda)F(u) {\rm d}v_g.$$ It is easy to see that the functional $\mathscr{J}_\lambda$ is well-defined and of class $C^1$ on $H^1_g(M)\times H^1_g(M)$. Moreover, due to relations
(\ref{schrodingerweak}) and (\ref{maxwellweak}) the pair $(u,\phi)\in H^1_g(M)\times
H^1_g(M)$ is a weak solution of $(\mathcal{SM}_{\lambda}^e)$ if and only if $(u,\phi)$ is a
critical point of $\mathscr{J}_\lambda$.

Using the Lax-Milgram theorem one can see that the equation $$-\Delta_g \phi+\phi=qu^2, \ \mbox{ in }M$$ has a unique solution for any fixed $u$. By exploring an idea of Benci and Fortunato \cite{BF}, we introduce the map $\phi_u:H_g^1(M)\to H_g^1(M)$ by associating to
every $u\in H^1_g(M)$ the unique solution $\phi=\phi_u$ of the
Maxwell equation. Thus, one can define the "one-variable" energy
functional $\mathcal{E}_{\lambda}:H_g^1(M)\to \mathbb{R}$ associated
with  system $(\mathcal{SM}_\lambda^e)$:
\begin{equation}\label{one-dim-energy}
\mathcal{E}_{\lambda}(u)=\frac{1}{2}\|u\|_{\beta}^2+\frac{e}{4}\int_M
\phi_u u^2 {\rm d}v_g-\ds\mathcal{F}(u),
\end{equation}
where $\mathcal F:H^1_g(M)\to \mathbb R$ is the
functional defined by $$\ds\mathcal{F}(u)=\int_M \Psi(x,\lambda)F(u)\d.$$ 
By using standard variational arguments, one has that the pair $(u,\phi)\in H^1_g(M)\times
H^1_g(M)$ is a critical point of $\mathscr{J}_\lambda$ if and only
if $u$ is a critical point of $\mathcal{E}_{\lambda}$ and
$\phi=\phi_u$, see for instance \cite{FK16}. Moreover, we  have that
\begin{equation}
\label{derivalt} \mathcal{E}_{\lambda}'(u)(v)=\int_M(\langle\nabla_g
u,\nabla_g v\rangle+\beta(x)uv+e\phi_u uv)\d-\int_M
\Psi(x,\lambda)f(u)v\d.\end{equation}

%

\subsection{Schr\"odinger-Maxwell systems involving sublinear nonlinearity}
In this section we set $\Psi(x,\lambda)=\lambda\alpha(x)+\mu_0\beta(x)$.
Recall that  

$$
\mathcal{E}_{\lambda}(u)=\frac{1}{2}\|u\|_{\beta}^2+\frac{e}{4}\int_M
\phi_u u^2 {\rm d}v_g-\int_M \Psi(x,\lambda)F(u)\d.
$$
In order to apply variational methods, we prove some elementary properties of the functional $\mathcal{E}_\lambda$:\newpage
\begin{lemma}\label{coerciv}
	The energy functional $\mathcal{E}_{\lambda}$ is coercive, for every
	$\lambda\geq0$. 
\end{lemma}

\begin{proof}
	Indeed, due to (\hyperref[f2-felt]{$f_{2}$}), we have that for every $\varepsilon>0$
	there exists $\delta>0$ such that $|F(s)|\leq\varepsilon|s|^{2}$,
	for every $|s|>\delta.$ Thus, since $\Psi(x,\lambda)\in L^\infty(M)$ we have that

	\begin{align*}
	\mathcal{F}(u)  &=\int_{\{u>\delta\}}\Psi(x,\lambda)F(u)\d+\int_{\{u\leq\delta\}}\Psi(x,\lambda)F(u)\d\\&\leq\varepsilon\left\Vert {\Psi(\cdot,\lambda)}\right\Vert _{L^{\infty}(M)}\kappa_2^2\|u\|_{\beta}^{2}+\|\Psi(\cdot,\lambda)\|_{L^{\infty}(M)}\mathrm{Vol}_g M\max_{|s|\leq\delta}|F(s)|.
	\end{align*}
	Therefore, 
	\[
	\mathcal{E}_{\lambda}(u)\geq\left(\frac{1}{2}-\varepsilon\kappa_2^2\left\Vert {\Psi(\cdot,\lambda)}\right\Vert _{L^{\infty}(M)}\right)\|u\|_{\beta}^{2}-\mathrm{Vol_{g}}M\cdot \|\Psi(\cdot,\lambda)\|_{L^{\infty}(M)}\max_{|s|\leq\delta}|F(s)|.
	\]
	In particular, if $0<\varepsilon<(2\kappa_2^2\|\Psi(\cdot,\lambda)\|_{L^{\infty}(M)})^{-1}$,
	then $\mathcal{E}_{\lambda}(u)\to\infty$ as $\|u\|_{\beta}\to\infty$.
\end{proof}

\begin{lemma}\label{PS}
	The energy functional $\mathcal{E}_{\lambda}$ satisfies the Palais-Smale
	condition for every $\lambda\geq0$. 
\end{lemma}

\begin{proof}
	Let $\{u_{j}\}_{j}\subset H_{g}^{1}(M)$ be a Palais-Smale sequence,
	i.e., $\{\mathcal{E}_{\lambda}(u_{j})\}_{j}$ is bounded and
	$$\Vert(\mathcal{E}_{\lambda})^{\prime}(u_{j})\Vert_{H_{g}^{1}(M)^{\ast}}\rightarrow0$$
	as $j\rightarrow\infty.$ Since $\mathcal{E}_{\lambda}$ is coercive (see Lemma \ref{coerciv}),
	the sequence $\{u_{j}\}_{j}$ is bounded in $H_{g}^{1}(M)$. Therefore,
	up to a subsequence, then $\{u_{j}\}_{j}$ converges weakly in $H_{g}^{1}(M)$
	and strongly in $L^{p}(M)$, $p\in(2,2^{*}),$ to an element $u\in H_{g}^{1}(M)$.

	First we claim that, $\hbox{for all }u,v\in H^1_g(M)$ we have that \begin{equation}\label{fontos11}\ds \int\limits_M \left(u\phi_u-v\phi_v\right)(u-v)\d\geq 0.\end{equation} This inequality is equivalent with the following one: $$\int\limits_M \phi_u u^2 \d+\int\limits_M \phi_v v^2 \d\geq \int\limits_M (\phi_u uv+\phi_vuv)\d.$$ On the other hand using the Cauchy-Schwarz inequality, we have, that \begin{align*}
	\int\limits_M (\phi_u uv+\phi_vuv)\d&\leq \left(\int\limits_M\phi_u u^2 \d\right)^{1/2} \left(\int\limits_M\phi_u v^2 \d\right)^{1/2}+\left(\int\limits_M\phi_v u^2 \d\right)^{1/2}\left(\int\limits_M\phi_v v^2 \d\right)^{1/2}\\ &=\frac{1}{q}\left(\int\limits_M (\nabla_g\phi_u\nabla_g \phi_v+\phi_u\phi_v)\d\right)^{1/2}(\|\phi_u\|_{H^1_g(M)}+\|\phi_v\|_{H^1_g(M)}) \\ & \leq \frac{1}{q}\|\phi_u\|^{1/2}_{H^1_g(M)}\|\phi_v\|^{1/2}_{H^1_g(M)}\left(\|\phi_u\|_{H^1_g(M)}+\|\phi_v\|_{H^1_g(M)}\right).
	\end{align*} 
	Taking into account the following algebraic inequality $ (xy)^{1/2}(x+y)\leq (x^2+y^2),(\forall)x,y\geq 0$, we have that $$\|\phi_u\|^{1/2}_{H^1_g(M)}\|\phi_v\|^{1/2}_{H^1_g(M)}\left(\|\phi_u\|_{H^1_g(M)}+\|\phi_v\|_{H^1_g(M)}\right)\leq \|\phi_u\|^2_{H^1_g(M)}+\|\phi_v\|^2_{H^1_g(M)}.$$ Therefore, $$\int\limits_M (\phi_u uv+\phi_vuv)\d\leq \frac{1}{q} \left(\|\phi_u\|^2_{H^1_g(M)}+\|\phi_v\|^2_{H^1_g(M)}\right)= \int\limits_M \phi_u u^2 \d+\int\limits_M \phi_v v^2 \d,$$
	which proves the claim.

	Now, using inequality \eqref{fontos11} one has 
	\[
	\int_{M}|\nabla_{g}u_{j}-\nabla_{g}u|^{2}\d+\int_{M}\beta(x)\left(u_{j}-u\right)^{2}\d\leq
	\]
	\[
	(\mathcal{E}_{\lambda})^{\prime}(u_{j})(u_{j}-u)+(\mathcal{E}_{\lambda})^{\prime}(u)(u-u_{j})+\int_{M}\Psi(x,\lambda)[f(u_{j}(x))-f(u(x))](u_{j}-u)\d.
	\]
	Since $\Vert(\mathcal{E}_{\lambda})^{\prime}(u_{j})\Vert_{H_{g}^{1}(M)^{\ast}}\rightarrow0$,
	and $u_{j}\rightharpoonup u$ in $H_{g}^{1}(M)$, the first two terms
	at the right hand side tend to $0$. Let $p\in(2,2^{*}).$ 
	
	By the
	assumptions on $f$, for every $\varepsilon>0$ there exists a constant $C_{\varepsilon}>0$
	such that $$|f(s)|\leq\varepsilon|s|+C_{\varepsilon}|s|^{p-1},$$ for
	every $s\in\mathbb{R}$. The latter relation, H\"older inequality and
	the fact that $u_{j}\rightarrow u$ in $L^{p}(M)$ imply that 
	\[
	\left\vert \int_{M}\Psi(x,\lambda)[f(u_{j})-f(u)](u_{j}-u)\d\right\vert \rightarrow0,
	\]
	as $j\to\infty.$ Therefore, $\Vert u_{j}-u\Vert_{H_{g}^{1}(M)}^{2}\to0$
	as $j\to\infty$, which proves our claim.
\end{proof}
Before we prove Theorem \ref{thm:sublinear1} we prove the following lemma:
\begin{lemma} Let $f:[0,+\infty)\to \mathbb{R}$ be a continuous function satisfying the assumptions (\hyperref[f1-felt]{$f_1$})-(\hyperref[f3-felt]{$f_3$}).
Then $$\ds c_f:=\max_{s>0}\frac{f(s)}{s}>c_F:=\max_{s>0}\frac{4F(s)}{2s^2+eqs^4}.$$
\end{lemma} 
\begin{proof}
	Let $s_0>0$ be a maximum point for the function $\ds s\mapsto \frac{4F(s)}{2s^2+eqs^4}$, therefore $$c_F=\frac{4F(s_0)}{2s_0^2+eqs_0^4}=\frac{f(s_0)}{s_0+eqs_0^3}\leq \frac{f(s_0)}{s_0}\leq c_f.$$ Now we assume that $c_f=c_F:=\theta$. Let $$\widetilde{s}_0:=\inf\left\{s>0:\theta=\frac{4F(s)}{2s^2+eqs^4}\right\}.$$ Note that $\widetilde{s}_0>0$. Fix $t_0\in (0,\widetilde{s}_0)$, in particular $4F(t_0)<\theta(2t_0^3+eqt_0^4)$. On the other hand, from the definition of $c_f$, one has $f(t)\leq \theta(s+eqs^3)$. Therefore $$0=4F(\widetilde{s}_0)-\theta(2\widetilde{s}_0+eq\widetilde{s}_0^4)=\left(4F(t_0)-\theta(2t_0^2+eqt_0^4)\right)+4\int\limits_{t_0}^{\widetilde{s}_0} \left(f(t)-\theta(s+eqs^3)\right)ds<0,$$ which is a contradiction, thus $c_f>c_F$. 
\end{proof}

Now we are in the position to prove Theorem \ref{thm:sublinear1}.
\begin{proof}[Proof of Theorem \ref{thm:sublinear1}] First recall that, in this case, $\beta(x)\equiv 1$ and $\Psi(\lambda,x)=\lambda\alpha(x)$, and $\alpha\in C^{\infty}(M)$ is a positive function.

	\noindent (a) Let $\lambda\geq0.$ If we choose $v=u$ in (\ref{schrodingerweak})
	we obtain that 
	\[
	\int_{M}\left(|\nabla_{g}u|^{2}+u^{2}+e\phi_{u}u^{2}\right)\d=\lambda\int_{M}\alpha(x)f(u)u\d.
	\]
	As we already mentioned, due to the assumptions (\hyperref[f1-felt]{$f_1$})-(\hyperref[f3-felt]{$f_3$}), the number $\ds c_{f}=\max_{s>{0}}\frac{f(s)}{s}$
	is well-defined and positive. Thus, since $\ds \|\phi_u\|_{H^1_g(M)}^2=q\int_M \phi_{u} u^2\d\geq 0$,  we have
	that 
	
	$$\ds \|u\|_{H^1_g(M)}^2\leq \|u\|_{H^1_g(M)}^2+e\int_M \phi_{u} u^2\d\leq \lambda c_f \|\alpha\|_{L^\infty(M)}\int_M u^2 \d\leq \lambda c_f \|\alpha\|_{L^\infty(M)}\|u\|_{H^1_g(M)}^2.$$
	
	Therefore, if $\lambda<c_{f}^{-1}\left\Vert \alpha\right\Vert _{L^{\infty}(M)}^{-1}$,
	then the last inequality gives $u=0$. By the Maxwell's equation we
	also have that $\phi=0$, which concludes the proof of (a).
	
	\noindent (b) By using assumptions $(f_{1})$ and $(f_{2})$, one
	has 
	\[
	\lim_{\mathscr{H}(u)\to0}\frac{\mathcal{F}(u)}{\mathscr{H}(u)}=\lim_{\mathscr{H}(u)\to\infty}\frac{\mathcal{F}(u)}{\mathscr{H}(u)}=0,
	\]
	where ${\displaystyle \mathscr{H}(u)=\frac{1}{2}\|u\|_{\beta}^{2}+\frac{e}{4}\int_{M}\phi_{u}u^{2}\d}$.
	Since $\alpha\in C^{\infty}(M)_{+}\setminus\{0\},$ on
	account of ($f_{3}$), one can guarantee the existence of a suitable
	truncation function $u_{T}\in H_{g}^{1}(M)\setminus\{0\}$ such
	that $\mathcal{F}(u_{T})>0.$ Therefore, we may define 
	\[
	{\displaystyle \lambda_{0}=\inf_{\ds\substack{u\in H_{g}^{1}(M)\setminus\{0\}\\
				\mathcal{F}(u)>0
			}
		}\frac{\mathscr{H}(u)}{\mathcal{F}(u)}.}
	\]
	The above limits imply that $0<\lambda_{0}<\infty.$
	Since $H^1_g(M)$ contains the positive constant functions on $M$, we have $$\lambda_0=\inf_{\ds\substack{u\in H_{g}^{1}(M)\setminus\{0\}\\
			\mathcal{F}(u)>0
		}
	}\frac{\mathscr{H}(u)}{\mathcal{F}(u)}\leq  \max_{s>0}\frac{2s^2+eqs^4}{4F(s)\|\alpha\|_{L^1(M)}}= c_F^{-1}\|\alpha\|_{L^1(M)}^{-1}.$$
	
	For every $\lambda>\lambda_{0}$,
	the functional $\mathcal{E}_{\lambda}$ is bounded from below, coercive
	and satisfies the Palais-Smale condition (see Lemma \ref{coerciv}, Lemma \ref{PS}). If we fix $\lambda>\lambda_{0}$
	one can choose a function $w\in H_{g}^{1}(M)$ such that $\mathcal{F}(w)>0$
	and $$\lambda>\frac{\mathscr{H}(w)}{\mathcal{F}(w)}\geq\lambda_{0}.$$

	In particular, $$\ds c_{1}:=\inf_{H_{g}^{1}(M)}\mathcal{E}_{\lambda}\leq\mathcal{E}_{\lambda}(w)=\mathscr{H}(w)-\lambda\mathcal{F}(w)<0.$$
	The latter inequality proves that the global minimum $u_{\lambda}^{1}\in H_{g}^{1}(M)$
	of $\mathcal{E}_{\lambda}$ on $H_{g}^{1}(M)$ has negative energy
	level. 
	
	In particular, $(u_{\lambda}^{1},\phi_{u_{\lambda}^{1}})\in H_{g}^{1}(M)\times H_{g}^{1}(M)$
	is a nontrivial weak solution to $(\mathcal{SM}^e_{\lambda})$. 
	
	\noindent Let $\nu\in(2,6)$ be fixed. By assumptions, for any $\eps>0$
	there exists a constant $C_{\eps}>0$ such that 
	\[
	0\leq|f(s)|\leq\frac{\eps}{\|\alpha\|_{L^{\infty}(M)}}|s|+C_{\eps}|s|^{\nu-1}\hbox{ for all}\ s\in\mathbb{R}.
	\]
	Thus
	
	\begin{align*}0\leq |\mathcal{F}(u)| &\leq \int_M
	\alpha(x)|F(u(x))|\d \\ &\leq \int_M
	\alpha(x)\left(\frac{\eps}{2\|\alpha\|_{L^\infty(M)}}u^2(x)+\frac{C_\eps}{\nu}|u(x)|^\nu\right)\d
	\\ &\leq
	\frac{\eps}{2}\|u\|^2_{H^1_g(M)}+\frac{C_\eps}{\nu}\|\alpha\|_{L^\infty(M)}
	\widetilde{\kappa}_\nu^\nu \|u\|^\nu_{H^1_g(M)},\end{align*}
	where $\widetilde{\kappa}_\nu$ is the embedding constant in the compact embedding $H^1_g(M)\hookrightarrow L^\nu(M),\ \nu\in [1,6)$.
	
	Therefore, 
	\[
	\mathcal{E}_{\lambda}(u)\geq\frac{1}{2}(1-\lambda\eps)\|u\|_{H^1_g(M)}^{2}-\frac{\lambda C_{\eps}}{\nu}\|\alpha\|_{L^{\infty}(M)}\widetilde{\kappa}_{\nu}^{\nu}\|u\|_{H^1_g(M)}^{\nu}.
	\]
	Bearing in mind that $\nu>2$, for enough small $\rho>0$ and $\eps<\lambda^{-1}$
	we have that 
	\[
	\ds\inf_{\|u\|_{H^1_g(M)}=\rho}\mathcal{E}_{\lambda}(u)\geq\frac{1}{2}\left(1-\eps\lambda\right)\rho-\frac{\lambda C_{\eps}}{\nu}\|\alpha\|_{L^{\infty}(M)}\widetilde{\kappa}_{\nu}^{\nu}\rho^{\frac{\nu}{2}}>0.
	\]
	A standard mountain pass argument (see for instance, Willem \cite{Willem}) implies
	the existence of a critical point $u_{\lambda}^{2}\in H_{g}^{1}(M)$
	for $\mathcal{E}_{\lambda}$ with positive energy level. Thus $(u_{\lambda}^{2},\phi_{u_{\lambda}^{2}})\in H_{g}^{1}(M)\times H_{g}^{1}(M)$
	is also a nontrivial weak solution to $(\mathcal{SM}^e_{\lambda})$.
	Clearly, $u_{\lambda}^{1}\neq u_{\lambda}^{2}$. 
\end{proof}

It is also clear that the function $\ds q\mapsto \max_{s>0}\frac{4F(s)}{2s^2+eqs^4}$ is non-increasing. Let $a>1$ be a real number. Now, consider the following function

\[
f(s)=\begin{cases}
0, & 0\leq s<1,\\
s+g(s), & 1\leq s<a,\\
a+g(a), & s\geq a,
\end{cases}
\]
where $g:[1,+\infty)\to \mathbb{R}$ is a continuous function with the following properties
\begin{itemize}
	\item[$(g_1)$] $g(1)=-1$;
	\item[$(g_2)$] the function  $\ds s\mapsto\frac{g(s)}{s}$ is non-decreasing on $[1,+\infty)$;
	\item[$(g_3)$] $\ds\lim_{s\to \infty}g(s)<\infty$.
\end{itemize}
In this case the \[
F(s)=\begin{cases}
0, & 0\leq s<1,\\
\ds \frac{s^{2}}{2}+G(s)-\frac{1}{2}, & 1\leq s<a,\\
\ds (a+g(a))s-\frac{a^{2}}{2}+G(a)-ag(a)-\frac{1}{2}, & s\geq a,
\end{cases}
\]
where $\ds G(s)=\int_{1}^{s}g(t)dt$. It is also clear that $f$ satisfies the assumptions $(\hyperref[f1-felt]{f_1})-(\hyperref[f3-felt]{f_3})$.

Thus, a simple calculation shows that $$c_f=\frac{a+g(a)}{a}.$$ We also claim that $$\widehat{c}_F=\lim_{q\to 0}c_F=\frac{(a+g(a))^2}{a^2+2ag(a)-2G(a)+1}.$$
Indeed, $$\widehat{c}_F=\max_{s>0}\frac{2F(s)}{s^2}.$$ It is clear that, it is enough to show that the maximum of the function $\ds \frac{2F(s)}{s^2}$ is achieved on the interval $s\geq a$, i.e., $$sg(s)-2G(s)> -1,\ s>1.$$ Now, using a result of \cite[page 42, equation (4.3)]{Gromov} (see also \cite[Theorem 1.3]{Gromov-eredmeny}), we have that the function $\ds \frac{G(s)}{\frac{s^2-1}{2}}$ is increasing, thus $$sg(s)-2G(s)\geq \frac{g(s)}{s}\geq -1, \ s\geq -1,$$ which proves our claim.

One can see that, from the assumptions on $g$, that the values $c_f$ and $\widehat{c}_F$ may be arbitrary close to each other. Indeed, when $$\ds \lim_{a\to \infty} c_f=\lim_{a\to \infty}\widehat{c}_F=1.$$ Therefore, if $\alpha\equiv 1$ then the threshold values are $c_f^{-1}$ and $c_F^{-1}$ (which are constructed independently), i.e. if $\lambda \in (0,c_f^{-1})$ we have just the trivial solution, while if $\lambda \in (c_F^{-1},+\infty)$ we have at least two solutions. $\lambda$ lying in the
gap-interval $[c_f^{-1},c_F^{-1}]$ we have no information on the number of solutions for $(\mathcal{SM}^e_{\lambda})$.

Taking into account the above example we see that if the "impact" of the Maxwell equation is small ($q\to 0$), then the values $c_f$ and $c_F$ may be arbitrary close to each other. 
\begin{remark}
	Typical examples for function $g$ can be:
	\begin{itemize}
		\item[(a)] $g(s)=-1$. In this case $c_f=\frac{a-1}{a}\ \ \mbox{and}\ \ \widehat{c}_F=\frac{a-1}{a+1}.$ 
		\item[(b)] $g(s)=\frac{1}{s}-2$. In this case $c_f =\frac{(a-1)^2}{a^2}\ \ \mbox{and}\ \ \widehat{c}_F=\frac{(a-1)^4}{a^2(a^2-2\ln a-1)}$.
	\end{itemize}
\end{remark}

\begin{proof}[Proof of Theorem \ref{thm:m-megoldas}]  We follow the idea presented in \cite{Kristaly-Studia}.	First we claim that the set of all global minima's of the functional
	$\mathcal{N}:H_{g}^{1}(M)\to\mathbb{R}$,
	\[
	\mathcal{N}(u)=\frac{1}{2}\|u\|_{\beta}^{2}-\mu_{0}\int_{M}\beta(x)F(u)\d
	\]
	has at least $m$ connected components in the weak topology on $H_{g}^{1}(M).$
	Indeed, for every $u\in H_{\beta}^{1}(M)$ one has 
	\begin{align*}
	\mathcal{N}(u) & =\frac{1}{2}\|u\|_{\beta}^{2}-\mu_{0}\int_{M}\beta(x)F(u)\d\\
	& =\frac{1}{2}\int_{M}|\nabla_{g}u|^{2}\d+\int_{M}\beta(x)\Phi_{\mu_{0}}(u)\d\\
	& \geq\|\beta\|_{L^1(M)}\inf_{t}\Phi_{\mu_{0}}(t).
	\end{align*}
	Moreover, if we consider $u=\tilde{t}$ for a.e. $x\in M$, where
	$\tilde{t}\in\mathbb{R}$ is the minimum point of the function $t\mapsto\Phi_{\mu_{0}}(t)$,
	then we have equality in the previous estimate. Thus,
	\[
	\inf_{u\in H_{\beta}^{1}(M)}\mathcal{\mathcal{N}}(u)=\|\beta\|_{L^1(M)}\inf_{t}\Phi_{\mu_{0}}(t).
	\]
	On the other hand if $u\in H_{g}^{1}(M)$ is not a constant function,
	then $|\nabla_{g}u|^{2}>0$ on a positive measure set in $M,$ i.e.,
	$$\mathcal{N}(u)>\|\beta\|_{L^1(M)}\inf_{t}\Phi_{\mu_{0}}(t).$$ Consequently,
	there is a one-to-one correspondence between the sets
	\[
	\mathrm{Min}(\mathcal{N})=\left\{ u\in H_{g}^{1}(M):\ \mathcal{N}(u)=\inf_{u\in H_{g}^{1}(M)}\mathcal{N}(u)\right\} 
	\]
	and
	\[
	\mathrm{Min}\left(\Phi_{\mu_{0}}\right)=\left\{ t\in\mathbb{R}:\ \Phi_{\mu_{0}}(t)=\inf_{t\in\mathbb{R}}\Phi_{\mu_{0}}(t)\right\} .
	\]
	Let $\xi$ be the function that associates to every $t\in\mathbb{R}$
	the equivalence class of those functions which are a.e. equal to $t$
	on the whole $M.$ Then $\xi:\mathrm{Min}(\mathcal{N})\to\mathrm{Min}\left(\Phi_{\mu_{0}}\right)$
	is actually a homeomorphism, where $\mathrm{Min}(\mathcal{N})$ is
	considered with the relativization of the weak topology on $H_{g}^{1}(M).$
	On account of $(f_{4})$, the set $\mathrm{Min}\left(\Phi_{\mu_{0}}\right)$
	has at least $m\geq2$ connected components. Therefore, the same is
	true for the set $\mathrm{Min(\mathcal{N})}$, which proves the claim. 
	
	Now we are in the position to apply Theorem \ref{Ricceri-m} with
	$H=H_{g}^{1}(M)$, $\mathcal{N}$ and
	\[
	{\displaystyle \mathcal{G}=\frac{1}{4}\int_{M}\phi_{u}u^{2}\d-\int_{M}\alpha(x)F(u)\d}.
	\]
	
	Now we prove that the functional $\mathcal{G}$ is sequentially weakly lower semicontinuous. To see this, it is enough to prove that the map $$H^1_\beta(M)\ni u\mapsto \int_M \phi_{u} u^2\d$$ is convex. To prove this, let us fix $u,v\in H^1_{\beta}(M)$ and $t,s\geq 0$ such that $t+s=1$.  Then we have that 
	\begin{align*}\mathscr{A}(\phi_{tu+sv}):=-\Delta_g \phi_{tu+sv}+\phi_{tu+sv} &=q(tu+sv)^2 \\
	&\leq q(tu^2+sv^2) \\ &= t(qu ^2)+s(qv^2) \\ &= t(-\Delta_g \phi_u+\phi_u)+s(-\Delta_g \phi_v+\phi_v) \\ &=\mathscr{A}(t\phi_u+s\phi_v).\end{align*}  Then using a comparison principle it follows that $$\phi_{tu+sv}\leq t\phi_u+s\phi_v.$$ Then multiplying the equations $-\Delta_g \phi_u+\phi_u=qu^2$ by $\phi_v$ and $-\Delta_g \phi_v+\phi_v=qv^2$ by $\phi_u$, after integration, we obtain that 
	\begin{equation}
	\label{fontosossz}
	\int\limits_M (\nabla_g \phi_u \nabla_g \phi_v+\phi_u\phi_v )\d=q\int\limits_M u^2 \phi_v \d=q\int\limits_M v^2 \phi_u \d.\end{equation}
	Thus, combining the above outcomes we have
	\begin{align*}
	\int\limits_M \phi_{tu+sv}(tu+sv)^2 dv_g &\leq \int\limits_M \left(t\phi_u+s\phi_v\right)\left(tu^2+sv^2\right)dv_g \\ &=
	t^2 \int\limits_M \phi_u u^2 dv_g+ts\int\limits_M\left(\phi_u v^2+\phi_v u^2\right)dv_g+s^2\int\limits_M \phi_v v^2 dv_g\\ &\overset{\eqref{fontosossz}}{=} \frac{t^2}{q}\left( \int\limits_M |\nabla_g \phi_u|^2dv_g+\int\limits_M \phi_u^2 dv_g \right)+ \frac{2ts}{q}\int\limits_M (\nabla_g \phi_u \nabla_g \phi_v+\phi_u\phi_v )dv_g \\ &+\frac{s^2}{q}\left( \int\limits_M |\nabla_g \phi_v|^2dv_g+\int\limits_M \phi_v^2 dv_g \right) \\&=\frac{1}{q}\int\limits_M (t\nabla_g \phi_u+s\nabla_g \phi_v )^2 dv_g+\frac{1}{q}\int\limits_M (t\phi_u+s\phi_v )^2dv_g\\ &\leq t\int\limits_M \phi_u u^2dv_g+s\int\limits_M \phi_v v^2 dv_g,
	\end{align*}
	which gives the required inequality, therefore it follows the required convexity. Almost the same way as in Lemma \ref{PS} we can prove that $\mathcal{N}+\lambda\mathcal{G}$ satisfies the Palais-Smaile condition for every $\lambda>0$ small enough. Therefore the functionals $\mathcal{N}$ and $\mathcal{G}$
	satisfies all the hypotheses of Theorem \ref{Ricceri-m}. Therefore
	for every $\ds\tau>\max\left\{0,\|\beta\|_{1}\inf_{t}\Phi_{\mu_{0}}(t)\right\}$
	there exists $\lambda_{\tau}>0$ such that for every $\lambda\in(0,\lambda_{\tau})$
	the problem $(\mathcal{SM}_{\lambda}^\lambda)$ has at least $m+1$ solutions. We know in addition that $m$ elements among the solutions belong to the set $\ds \mathcal{N}_{\nu_0}^{-1}((-\infty,\tau))$, which proves that $m$ solutions satisfy the inequality $$\frac{1}{2}\int_M \left(|\nabla_{g} u|^2+\beta(x)u^2\right)\d-\mu_{0}\int_M \beta(x)F(u)\d<\tau.$$ 
\end{proof}
\begin{remark} \noindent (a) Note that (\hyperref[f4-felt]{$f_4$}) implies that the function $t\mapsto \Phi_{\mu_{0}}(t)$ has at least $m-1$ local
maxima. Thus, the function $t\mapsto \mu_0f(t)$ has at least $2m-1$ fixed points. In particular, if for some $\lambda>0$ $$\Psi(x,\lambda)=\mu_{0}\beta(x), \mbox{ for every }x\in M,$$ then the problem $(\mathcal{SM}_{\lambda}^\lambda)$ has at least $2m - 1 \geq 3$ constant solutions. 
\\ \noindent (b) Using the abstract Theorem \ref{Ricceri-m}, one can guarantee that $\ds \tau>\max\left\{0,\|\beta\|_{L^1(M)}\inf_t\Phi_{\mu_{0}}(t)\right\}$ It is clear that the assumption \hyperref[f2-felt]{($f_2$)} holds if there exist $\nu\in (0,1)$ and $c>0$ such that $$|f(t)|\leq c|t|^\nu,\ \mbox{ for every }t\in \mathbb{R}.$$ In this case, $m$ weak solutions of the problem satisfy the inequality $$\frac{1}{2}\int_M \left(|\nabla_{g} u|^2+\beta(x)u^2\right)\d-\mu_{0}\int_M \beta(x)F(u)\d<\tau.$$ Now, it is clear that $$|F(t)|\leq \frac{c}{\nu+1}|t|^{\nu+1}, \ \mbox{ for every }t\in \mathbb{R}.$$ Using a H\"older inequality $$\int_M \beta(x) |u|^{\nu+1}\d\leq \|\beta\|_{L^1(M)}^{\frac{1-\nu}{2}}\|u\|_{H^1_\beta(M)}^{\nu+1}.$$ One can observe, that since $\tau>0$ the equation $$\frac{1}{2}t^2-\frac{\mu_0c\|\beta\|_{L^1{M}}^{\frac{1-\nu}{2}}}{\nu+1}t^{\nu+1}-\tau=0,$$ always has a positive solution. 

Summing up, the number $\|u\|_{H^1_\beta(M)}$ is less than the greatest solution of the previous algebraic equation. Combining this with \eqref{normakhasonlitasa}, we have that $$\|u\|_{H^1_g(M)}\leq \frac{t_*}{\min\{1,\min_M\sqrt{\beta}\}},$$ where $t_*$ the greatest solution of the previous algebraic equation. A similar study for Emden--Fowler equation was done by Krist\'aly and R\v adulescu, see \cite[Theorem 1.3]{Kristaly-Studia}.
\end{remark}

\subsection{Schr\"odinger-Maxwell systems involving superlinear nonlinearity}
In the sequel we prove Theorem \ref{thm:szuperlinear}. Recall that $\Psi(\lambda,x)=\lambda$ and $\beta\equiv 1$. The energy functional associated with the problem $(\mathcal{SM}_\lambda^e)$ is defined by $$\mathcal{E}_\lambda(u)=\frac{1}{2}\|u\|_{H^1_g(M)}^2+\frac{e}{4}\int_M \phi_u u^2\d-\lambda\int_{M} F(u)\d.$$
\begin{lemma}
	Every (PS) sequence for the functional $\mathcal{E}_{\lambda}$ is
	bounded in $H_{g}^{1}(M).$
\end{lemma}

\begin{proof}
	We consider a Palais-Smale sequence $(u_{j})_{j}\subset H_{g}^{1}(M)$
	for $\mathcal{E}_{\lambda}$, i.e., $\{\mathcal{E}_{\lambda}(u_{j})\}$
	is bounded and $$\Vert(\mathcal{E}_{\lambda})^{\prime}(u_{j})\Vert_{H_{g}^{1}(M)^{\ast}}\rightarrow0\
	\mbox{as}\ j\rightarrow\infty.$$ We claim that $(u_{j})_{j}$ is bounded
	in $H_{g}^{1}(M).$ We argue by contradiction, so suppose the contrary.
	Passing to a subsequence if necessary, we may assume that
	\[
	\|u_{j}\|_{H_{g}^{1}(M)}\to\infty,\ \ \mbox{as }j\to\infty.
	\]
	It follows that there exists $j_{0}\in\mathbb{N}$ such that for every $j\geq j_0$ we have that
	\begin{align*}
	\mathcal{E}_{\lambda}(u_{j})-\frac{\left\langle \mathcal{E}_{\lambda}'(u_{j}),u_{j}\right\rangle }{\eta} &=\frac{1}{2}\left(\frac{\eta-2}{\eta}\right)\|u_{j}\|_{H_{g}^{1}(M)}^{2}+\frac{e}{4}\left(\frac{\eta-4}{\eta}\right)\int_{M}\phi_{u_{j}}u_{j}^{2}\d\\&+\lambda\int_{M}\left(\frac{f(u_{j})u_{j}}{\eta}-F(u_{j})\right)\d.
	\end{align*}
	Thus, bearing in mind that $\ds \int_M \phi_u u^2\d\geq 0$ and $(\hyperref[f222-felt]{\tilde{f}_{2}})$ one has that 
	\[
	\frac{1}{2}\left(\frac{\eta-2}{\eta}\right)\|u_{j}\|_{H_{g}^{1}(M)}^{2}\leq\mathcal{E}_{\lambda}(u_{j})-\frac{\left\langle \mathcal{E}_{\lambda}'(u_{j}),u_{j}\right\rangle }{\eta}+\chi\mathrm{Vol_{g}}(M),
	\]
	where 
	\[
	\chi=\sup\left\{ \left|\frac{tf(t)}{\eta}-F(t)\right|:t\leq\tau_{0}\right\} .
	\]
	Therefore, for every $j\geq j_0$ we have that
	\[
	\frac{1}{2}\left(\frac{\eta-2}{\eta}\right)\|u_{j}\|_{H_{g}^{1}(M)}^{2}\leq\mathcal{E}_{\lambda}(u_{j})+\frac{1}{\eta}\Vert(\mathcal{E}_{\lambda})^{\prime}(u_{j})\Vert_{H_{g}^{1}*}\|u_{j}\|_{H_{g}^{1}(M)}+\chi\mathrm{Vol_{g}}(M).
	\]
	Dividing by $\|u_{j}\|_{H_{g}^{1}(M)}$ and letting $j\to\infty$ we get a contradiction,
	which implies the boundedness of the sequence $\{u_{j}\}_{j}$ in
	$H_{g}^{1}(M).$ 
\end{proof}
\begin{proof}[Proof of the Theorem \ref{thm:szuperlinear}]
	~ Let us consider as before the following functionals: 
	\[
	\mathscr{H}(u)=\frac{1}{2}||u||_{H_{g}^{1}(M)}^{2}+\frac{e}{4}\int_{M}\phi_{u}u^{2}\d\mbox{ and }\mathcal{F}(u)=\int_{M}F(u)\d.
	\]
	Form the positivity and the convexity of functional $u\mapsto\ds \int_M \phi_{u} u^2$ it follows that the functional $\mathscr{H}$ is sequentially weakly semicontinuous and coercive functional. It is also clear that $\mathcal{F}$ is sequentially weakly continuous.
	Then, for $\mu=\frac{1}{2\lambda},$ we define the functional $J_{\mu}(u)=\mu\mathscr{H}(u)-\mathcal{F}(u).$
	Integrating, we get from $(\hyperref[f222-felt]{\tilde{f}_{2}})$ that, $$F(ts)\geq t^{\eta}F(s), \ \ t\geq1 \mbox{ and }|s|\geq \tau_0 .$$ Now, let us consider a fixed function
	$u_{0}\in H_{g}^{1}(M)$ such that $$\mathrm{Vol_{g}}\left(\{x\in M: |u_0(x)|\geq \tau_0\}\right)>0,$$ and using the previous inequality and the fact that $\phi_{tu}=t^2\phi_u$, we have
	that: 
	\begin{align*}
	J_{\mu}(tu_{0}) & =\mu\mathscr{H}(tu_{0})-\mathcal{F}(tu_{0})\\
	& =\mu\frac{t^{2}}{2}||u_{0}||_{H_{g}^{1}(M)}^{2}+\mu\frac{e}{4}t^{4}\int_{M}\phi_{u_{0}}u_{0}^{2}-\int_{M}F(tu_{0})\\
	& \leq\mu t^{2}||u_{0}||_{H_{g}^{1}(M)}^{2}+\mu\frac{e}{2}t^{4}\int_{M}\phi_{u_{0}}u_{0}^{2}-t^{\eta}\int_{\{x\in M::|u_{0}|\geq\tau_{0}\}}F(u_{0})+\chi_{2}\mathrm{Vol}_{g}(M)\overset{\eta>4}{\rightarrow}-\infty,
	\end{align*}
	as $t\to\infty$, where 
	\[
	\chi_{2}=\sup\left\{ |F(t)|:|t|\leq\tau_{0}\right\} .
	\]
	Thus, the functional $J_{\mu}$ is unbounded from below. A similar
	argument as before shows that (taking eventually a subsequence), one
	has that the functional $J_{\mu}$ satisfies the (PS) condition. 
	
	Let us denote by $K_{\tau}=\left\{ x\in M:\|u\|_{H_{g}^{1}(M)}^{2}<\tau\right\} $
	and by 
	\[
	h(\tau)=\inf_{u\in K_{\tau}}\frac{\ds \sup_{v\in K_{\tau}}\mathcal{F}(v)-\mathcal{F}(u)}{\tau-\mathscr{H}(u)}
	\]
	Since $0\in K_{\tau}$, we have that 
	\[
	h(\tau)\leq\frac{\sup_{v\in K_{\tau}}\mathcal{F}(v)}{\tau}.
	\]
	On the other hand bearing in mind the assumption $(\hyperref[f333-felt]{\tilde{f}_{1}})$,
	we have that, 
	\[
	\mathcal{F}(v)\leq C\|v\|_{H_{g}^{1}(M)}+\frac{C}{p}\kappa_{p}^{p}\|v\|_{H_{g}^{1}(M)}^{p}.
	\]
	Therefore 
	\[
	h(\tau)\leq\frac{C}{2}\tau^{\frac{1}{2}}+\frac{C\kappa_{p}^{p}}{p}\tau^{\frac{p-2}{2}}.
	\]
	Thus, if $$\lambda<
	\lambda_0:=\frac{p\tau^{\frac{1}{2}}}{2pC+2C\kappa_{p}^p\tau^{\frac{p-1}{2}}}$$ one has $\mu=\frac{1}{2\lambda}>h(\tau)$.
	Therefore, we are in the position to apply Ricceri's result, i.e.,
	Theorem \ref{Ricceritetel}, which concludes our proof. 
\end{proof}
\begin{remark}
	From the proof of Theorem \ref{thm:szuperlinear} one can see that $$\lambda_0\leq \frac{p}{2C}\max_{\tau>0}\frac{\tau^{\frac{1}{2}}}{p+\kappa_{p}^p\tau^{\frac{p-1}{2}}}.$$ Since $p>2$, one can see that, $\ds \max_{\tau>0}\frac{p\tau^{\frac{1}{2}}}{2pC+2C\kappa_{p}^p\tau^{\frac{p-1}{2}}}<\infty$.
\end{remark}
\subsection*{Acknowledgment} The author would like to express his gratitude to Professor Alexandru Krist\'aly for his useful comments and remarks. I am also very grateful to the anonymous Referee, for his/her thorough review and highly appreciate the comments and suggestions, which significantly contributed to improving the quality of the manuscript. Research supported by a grant KPI/IPC, Grant No. 13/13/17 May 2017.
%
%
%

\end{document}